\documentclass[12 pt]{article}
\usepackage{amsmath}
\usepackage{amsthm}
\usepackage{graphicx}
\usepackage{amssymb}

\numberwithin{equation}{section}

\def\ep{\epsilon}

\def\sub{\subset}

\newcommand{\R}{{\mathbb{R}}}

\newcommand{\Z}{{\mathbb{Z}}}

\newtheorem{ass}{Assumption}
\newtheorem{df}{Definition}

\newtheorem{lem}{Lemma}
\newtheorem{prop}{Proposition}
\newtheorem{rem}{Remark}

\newtheorem{thm}{Theorem}

\title{Long-term regularity of 2D gravity water waves}
\author{Fan Zheng \footnote{\textsc{Fan Zheng}, ICMAT, Calle Nicol\'as Cabrera 13--15, Campus Cantoblanco, Universidad Aut\'onoma de Madrid, 28049 Madrid Spain, fan.zheng@icmat.es.
\newline
The research of \textsc{Fan Zheng} was partially supported by ERC (European Research Council) under Grant 788250.}
}

\date{}

\begin{document}

\maketitle
\renewcommand{\thefootnote}{}
\renewcommand{\thefootnote}{\arabic{footnote}}

\begin{abstract}
The two dimensional gravity water wave problem concerns the motion of an incompressible fluid occupying half the 2D space and flowing under its own gravity. In this paper we study long-term regularity of solutions evolving from small but non-localized initial data.

Our main result is that if the $H^s$ norm of the initial data is $\ep$,
where $s \gg 1$ and $\ep \ll 1$, then the equation is wellposed at least for a time proportional to $\ep^{-4}$, improving on the $\ep^{-3}$ lifespan obtained in \cite{BeFePu,Wu2DL}. We also study period water waves and show a lifespan bridging the gap between non-periodic waves and waves with a period of 1.
\end{abstract}

\pagebreak
\tableofcontents
\pagebreak
\section{Introduction}

The gravity water wave equation describes the motion of an incompressible, irrotational, inviscid fluid occupying a region
\begin{equation}\label{Om-def}
\Omega(t) = \{(x, y): y < h(x, t)\}
\end{equation}
with a free boundary $\Gamma(t)$ that is the graph of $h(\cdot, t)$.
The fluid moves under the action of gravity, normalized to be of unit strength, and that of pressure, assumed to be zero on the boundary $\Gamma(t)$. The equation of motion is then
\begin{align}
\nabla_{x,y} \cdot v &= 0, \tag{incompressibility}\label{incomp}\\
\nabla_{x,y} \times v &= 0, \tag{irrotationality}\label{irrot}\\
v_t + v \cdot \nabla_{x,y}v &= -\nabla_{x,y}p - (0, 1), \tag{Euler equation}\label{Euler}\\
\partial_t + v \cdot \nabla_{x,y}&\text{ is tangent to $(t, \Gamma(t))$}
\tag{motion of the boundary}\label{boundary}
\end{align}
By \ref{incomp} we can write $v = \nabla_{x,y}\Psi$, where $\Psi$ is the velocity potential (up to a constant), which is harmonic by \ref{irrot} and is thus determined by its boundary value $\psi(x, t) = \Psi(x, h(x), t)$. The \ref{Euler} can then be recast in the Zakharov formulation (see Section 1.1.4 of \cite{La-rev}):
\begin{equation}\label{Zakharov}
\begin{cases}
h_t = G(h)\psi,\\
\psi_t = -h - \frac{1}{2}|\nabla\psi|^2 + \frac{(G(h)\psi + \nabla h \cdot \nabla\psi)^2}{2(1 + |\nabla h|^2)}
\end{cases}
\end{equation}
where $G(h)\psi = \sqrt{1 + |\nabla h|^2}\partial_n\Psi$ is the Dirichlet-to-Neumann operator. The energy
\begin{equation}\label{E-def}
E = \int \frac{1}{2} (\psi G(h)\psi + h^2)dx
\end{equation}
is conserved, see Section 6.3.1 of \cite{La-rev}. As the vorticity is transported by the flow as in the Euler equation, the flow remains irrotational if it is initially so.

\subsection{Background}
Due to space limit, we will include only a small part of the literature.
The reader is directed to the bibliography, especially \cite{BrGrNi-rev,IoPu-rev,La-rev}, for more references.

The study of water waves has its root in Newton \cite{Newton}, Stokes \cite{Stokes} and Levi-Civita \cite{LC}. Local wellposedness of gravity water waves in the Euclidean space was first shown by Nalimov \cite{Na} and Shinbrot \cite{Shin}, assuming the Taylor sign condition \cite{Tay}. Then Wu \cite{Wu2DL,Wu3DL} dropped this condition, only assuming that the interface is non self-intersecting. All these results are local, valid only for a time period inversely proportional to the size of the initial data.

Global wellposedness of gravity water waves in three dimensions for small data was shown by Germain--Masmoudi--Shatah \cite{GMS} and Wu \cite{Wu3DG}.
The same problem in two dimensions is harder, due to weaker decay of the solution. The first result in this vein is that of almost global wellposedness by Hunter--Ifrim--Tataru \cite{HuIfTa2DaG} and Wu \cite{Wu2DaG}, who showed a lifespan exponential in terms of the reciprocal of the size of the initial data. Later global wellposedness was obtained by Alazard--Delort \cite{AlDe}, Ionescu--Pusateri \cite{IoPu} and Ifrim--Tataru \cite{IfTa2DG}.

Of the results above, the local ones assume only unweighted Sobolev norms of the initiial data, while the global ones also presuppose that the initial data decays far away from the origin by requiring it to be also small in a weighted Sobolev space, which allows for $1/t$ decay of the solution, and with more careful analysis, closes the estimates needed for global existence. Without that assumption of locality, only part of the argument survives, giving a lifespan of $\ep^{-2}$ in three dimensions. In two dimensions additional integrability in the equations can be exploited to extend the lifespan to $\ep^{-3}$, both in the Euclidean case \cite{Wu2DLL} and in the periodic case \cite{BeFePu}. In \cite{Z} the author combined the energy estimates and Strichartz estimates to extend the lifespan of three dimensional water wavaes to an almost global one.

\subsection{Plan of action}
Our goal is to study the stability of the system (\ref{Zakharov}) with respect to a small perturbation around its trivial equilibrium $(h, \psi) = (0, 0)$. The first step is to linearize the equation. To that end we define some Fourier multipliers.
\begin{df}
\begin{equation}\label{La-def}
\widehat{|\nabla|u}(\xi) = |\xi|\hat u(\xi), \quad
\widehat{\Lambda u} = \widehat{|\nabla|^{1/2}u} = |\xi|^{1/2}\hat u(\xi).
\end{equation}
\end{df}

To the first order, $G(h)\psi = |\nabla|\psi$ (see Theorem 2.1.1 in \cite{AD}), so the linearization of (\ref{Zakharov}) is
\begin{equation}\label{linearization}
\begin{cases}
h_t = |\nabla|\psi,\\
\psi_t = -h,
\end{cases}
\end{equation}
whose eigenvectors $u_\pm = h \pm i|\nabla^{1/2}\psi$ evolve according to the equation
\begin{equation}\label{Ut}
\partial_tu_\pm = e^{\mp it\Lambda}u.
\end{equation}
This reveals the dispersive nature of the equation and suggests that it is amenable to $L^2$ energy estimates and $L^\infty$ decay estimates. Indeed,
energy estimates has been worked out in \cite{AD} and we only need to quote their results in Proposition \ref{EneEst} below. It roughly says
\begin{equation}\label{Hs-dt}
\frac{d}{dt}\|u\|_{H^s}^2 \lesssim \|u\|_{H^s}^2\|u\|_{C^r}^2.
\end{equation}
Thus the growth of the energy is controlled by the spacetime norm $\|u\|_{L_t^2C^r}$.

Now we focus on two dimensioanl water waves, the topic of this paper.
Since $h$, $\psi$ and $u$ are all functions of one variable,
the standard decay estimate implies that $\|u(t)\|_{L^\infty}$ decays like $t^{-1/2}$, provided that $u(0) \in L^1$. The point of this paper, however, is to drop this assumption and only presuppose Sobolev norms on the initial data. In this setting, $L^\infty$ decay estimates are replaced by Strichartz-type spacetime norm estimates, specifically the $L^4L^\infty$ norm in two dimensions, see Lemma \ref{Strichartz}. Then by H\"older's inequality,
\begin{equation}\label{L2L4}
\|u\|_{L_t^2([0,T])C^r}^2 \lesssim \sqrt T\|u\|_{L_t^4([0,T])C^r}^2
\lesssim \sqrt T\|u(0)\|_{H^s}^2.
\end{equation}
To close the energy estimate (\ref{Hs-dt}), we require the right-hand side of (\ref{L2L4}) to be bounded by 1, whence the lifespan $T \approx \|u(0)\|_{H^s}^{-4}$.

The dispersive estimate is done in a similar way: we need to close the estimate of the spacetime norm with two factors of the $L^\infty$ norm on the right-hand side. Since the nonlinear terms in (\ref{Zakharov}) are quadratic, a normal form transformation is applied to make them cubic, taking advantage of the fact that the dispersion relation does not give rise to three wave resonance, i.e.,
\begin{equation}\label{Phi-def}
\Phi_{\mu\nu}(\xi_1, \xi_2) = \sqrt{|\xi_1 + \xi_2|} - \sqrt{|\xi_1|} - \sqrt{|\xi_2|} \neq 0
\end{equation}
unless $\xi_1$ or $\xi_2$ or $\xi_1 + \xi_2 = 0$, in which case we can exploit the structure of the nonlinearity to show that it does not actually matter. Let $N$ denote the nonlinearity, see (\ref{N-def}). Then we can use the following Strichartz estimates:
\begin{equation}
\begin{aligned}
\|u\|_{L_t^4C^r} &\lesssim \|u(0)\|_{H^{r+3/8}} + \|N\|_{L_t^1H^{r+3/8}},\\ \|N\|_{L_t^1H^{r+3/8}} &\le \sqrt T\|N\|_{L_t^2H^{r+3/8}}
\lesssim \|u\|_{L_t^4C^r}^2\|u\|_{L_t^\infty H^s},
\end{aligned}
\end{equation}
to closed the estimate, obtaining a lifespan of $\ep^{-4}$, see Theorem \ref{Thm1}.

While the whole structure of the proof resembles that in \cite{Z}, here we aim not only to improve on known results on lifespans of two dimensional water waves, but also to show that the framework established in \cite{Z} is easily adaptable, and specifically, that the wealth of estimates already present in the literature, for example \cite{AD,GMS}, can be readily assembled to yield a short proof of previously inaccessible results.

It should be added however, that in the energy estimates, one can have three factors of $L^\infty$ norms on the right-hand side of (\ref{Hs-dt}),
using additional integrability in two dimensions \cite{BeFePu,Wu2DL},
but this extra saving does not easily carry over to dispersive estimates,
because the trivial four wave resonances ($\sqrt{|\xi_1|} + \sqrt{|\xi_2|} - \sqrt{|\xi_1|} - \sqrt{|\xi_2|} = 0$) lead to modified scattering \cite{IoPu}. Controlling this effect seems to require more regularity in the frequency space, i.e., weights in the physical space, which is out of the scope of this paper.

Last but not least, we also treat the case of $R$-periodic water waves,
and improve on previous results in the case when $R > \ep^{-2}$, see Theorem \ref{Thm2}.

\subsection{Main results}
\begin{thm}\label{Thm1}
Let $s > 17.5$. Then there is a constant $c > 0$ such that for any initial data $(h_0, \psi_0)$ such that $h_0 \in H^s$, $|\nabla|^{1/2}\psi_0 \in H^{s-1/2}$, $|\nabla|^{1/2}w_0 \in H^s$ and
\begin{equation}\label{initial}
\|h_0\|_{H^s} + \||\nabla|^{1/2}w_0\|_{H^s} = \ep < c
\end{equation}
then there is a solution $(h, \psi) \in C([0, T], H^s \times H^{s-1/2})$ with $|\nabla|^{1/2}w \in C([0, T], H^s)$, where $T = c/\ep^4$.
\end{thm}
\begin{rem}
$w$ is a quantity with similar estimates to $\psi$, see Definition \ref{Tfg-def}.
\end{rem}

\begin{thm}\label{Thm2}
Let $s > 17.5$. Then there is a constant $c > 0$ such that for any $R$-periodic initial data $(h_0, \psi_0)$ satisfying $h_0 \in H^s$, $|\nabla|^{1/2}\psi_0 \in H^{s-1/2}$, $|\nabla|^{1/2}w_0 \in H^s$ and (\ref{initial}), then there is a solution $(h, \psi) \in C([0, T], H^s \times H^{s-1/2})$ with $|\nabla|^{1/2}w \in C([0, T], H^s)$, where
\begin{equation}\label{lifespan}
T =
\begin{cases}
c\ep^{-3}, & 1 \le R \le \ep^{-2},\\
c\sqrt R\ep^{-2}, & \ep^{-2} < R \le \ep^{-4},\\
c\ep^{-4}, & R > \ep^{-4}.
\end{cases}
\end{equation}
\end{thm}

\subsection{Organization}
Section \ref{GhSec} collects some basic estimates of the Dirichlet-to-Neumann operator $G(h)\psi$. In Sections \ref{EneSec} and \ref{StrSec} we obtain the energy estimates and the Strichartz estimates.
Theorems \ref{Thm1} and \ref{Thm2} are shown in Sections \ref{Thm1prf} and \ref{Thm2prf}.

\section{Estimating the Dirichlet-to-Neumann map}\label{GhSec}
\begin{df}
For $\gamma \in \R$ let $C_*^\gamma$ denote the Besov space $B^\gamma_{\infty,\infty}$.
\end{df}

We need estimates on the Sobolev and Besov norms of $G(h)\psi$ and several related quantities.

\begin{df}[See (4.35)--(4.37) in \cite{Z}]
\begin{equation}\label{Gh-int}
\begin{aligned}
G(h)\psi &= |\nabla|\psi + \int_0^1 \partial_sG(sh)\psi ds,\\
\partial_sG(sh)\psi &= -G(sh)[hB(sh)\psi] - (hV(sh)\psi)',\\
B = B(h)\psi &= \frac{G(h)\psi + h'\psi'}{1 + h'^2},\\
V = V(h)\psi &= \psi' - h'B(h)\psi.
\end{aligned}
\end{equation}
\end{df}

\begin{lem}[Lemma 2.0.5 in \cite{AD}]\label{Gh-Cr}
Let $\gamma > 3$ be such that $2\gamma \notin \Z$. Then for all $(h, |\nabla|^{1/2}\psi) \in C_*^\gamma \times C_*^{\gamma-1/2}$ such that $\|h'\|_{C_*^{\gamma-1}} + \|h'\|_{C_*^{-1}}^{1/2}\|h'\|_{H^{-1}}^{1/2} \le c_r$, we have $\|G(h)\psi\|_{C_*^{\gamma-1}} + \|B\|_{C_*^{\gamma-1}} + \|V\|_{C_*^{\gamma-1}} \lesssim_r \||\nabla|^{1/2}\psi\|_{C_*^{\gamma-1/2}}$.
\end{lem}
\begin{rem}\label{Gh-CrRem}
As everything is linear in $\psi$ and $L^2 \cap C_*^{\gamma-1/2}$ is dense in $C_*^{\gamma-1/2}$, $\psi$ only needs to lie in the space where the right-hand side makes sense.
\end{rem}

\begin{lem}\label{Gh-Hs}
Let $s > 7/2$. Then for all $(h, |\nabla|^{1/2}\psi) \in H^s \times H^{s-1/2}$ with $\|h\|_{H^s} \le c_s$ we have $\|G(h)\psi\|_{H^{s-1}} + \|B\|_{H^{s-1}} + \|V\|_{H^{s-1}} \lesssim_s \||\nabla|^{1/2}\psi\|_{H^{s-1/2}}$.
\end{lem}
\begin{rem}\label{Gh-HsRem}
We will pick $\gamma \in (3, s - 1/2) \backslash \frac12 \Z$.
By the embedding $H^s \sub C_*^\gamma \sub C_*^0$, $\|h'\|_{C_*^{\gamma-1}} + \|h'\|_{C_*^{-1}}^{1/2}\|h'\|_{H^{-1}}^{1/2} \lesssim_s \|h\|_{H^s} \le c_s$ is also small.
\end{rem}
\begin{proof}
By Remark \ref{Gh-CrRem} we can assume that $\psi$ is Schwartz.
By Remark \ref{Gh-HsRem} we have the necessary smallness condition to apply Theorem 2.1.1 in \cite{AD} to get
\begin{equation}
\begin{aligned}
\|(G(h)\psi, B, V)\|_{H^{s-1}}
&\lesssim_{\|h\|_{C_*^\gamma}} \||\nabla|^{1/2}\psi\|_{C_*^{\gamma-1/2}}\|h\|_{H^s} + \||\nabla|^{1/2}\psi\|_{H^{s-1/2}}\\
&\lesssim \||\nabla|^{1/2}\psi\|_{H^{s-1/2}}
\end{aligned}
\end{equation}
because $\|h\|_{H^s}$ is small and $H^{s-1/2} \sub C_*^{\gamma-1/2}$.
\end{proof}

\begin{lem}\label{Gh-Cr2}
Let $\gamma > 3$ be such that $2\gamma \notin \Z$. Then for all $(h, |\nabla|^{1/2}\psi) \in C_*^{\gamma+1} \times C_*^{\gamma+1/2}$ such that $\|h'\|_{C_*^\gamma} + \|h'\|_{C_*^{-1}}^{1/2}\|h'\|_{H^{-1}}^{1/2} \le c_r$, we have $\|G(h)\psi - |\nabla|\psi\|_{C_*^{\gamma-1}} + \|B - |\nabla|\psi\|_{C_*^{\gamma-1}} + \|V - \psi'\|_{C_*^{\gamma-1}} \lesssim_r \|h\|_{C_*^\gamma}\||\nabla|^{1/2}\psi\|_{C_*^{\gamma+1/2}}$.
\end{lem}
\begin{proof}
This follows from the identities above and Lemma \ref{Gh-Cr}.
\end{proof}

\begin{lem}\label{Gh-Hs2}
Let $s, \mu, \gamma \in \R$ be such that $s - 1/2 > \gamma > 3$, $s \ge \mu \ge 3/2$ and $2\gamma \notin \Z$. Then for all $(h, |\nabla|^{1/2}\psi \in C_*^\gamma \times (C^{\gamma-1/2} \cap H^{\mu-3/2})$ such that $\|h\|_{H^s} \le c_s$, we have
\begin{equation}
\begin{aligned}
\|G(h)\psi - |\nabla|\psi\|_{H^{\mu-1}}
&\lesssim_{s,\mu,\gamma} \||\nabla|^{1/2}\psi\|_{C_*^{\gamma-1/2}}\|h\|_{H^s} + \|h\|_{C_*^\gamma}\||\nabla|^{1/2}\psi\|_{H^{\mu-3/2}},\\
\|B(h)\psi - |\nabla|\psi\|_{H^{\mu-1}}
&\lesssim_{s,\mu,\gamma} \||\nabla|^{1/2}\psi\|_{C_*^{\gamma-1/2}}\|h\|_{H^s} + \|h\|_{C_*^\gamma}\||\nabla|^{1/2}\psi\|_{H^{\mu-1/2}}.
\end{aligned}
\end{equation}
\end{lem}
\begin{proof}
By Remark \ref{Gh-CrRem} we can assume that $\psi$ is Schwartz.
By Remark \ref{Gh-HsRem} we have the necessary smallness condition to apply (2.5.1) in \cite{AD} to get the bound for $G(h)\psi$. To get the other bound, we also need the expression of $B = B(h)\psi$ in (\ref{Gh-int}), the Sobolev multiplication theorem and the smallness of $\|h\|_{H^s}$.
\end{proof}

\section{Energy estimates}\label{EneSec}
Here we collect the assumptions on which Chapters 1--3 of \cite{AD} are based.

Let $T > 0$. Let $s$, $\rho$ such that $s > \rho + 1 > 14$ and that $2\rho \notin \Z$.

\begin{df}[Definition A.1.2 of \cite{AD}]\label{Tfg-def}
Define $w = \psi - T_Bh$, where
\begin{equation}
\widehat{T_fg}(\xi) = \int_{\xi_1+\xi_2=\xi} \varphi(\xi_1, \xi_2)\hat f(\xi_1)\hat g(\xi_2)d\xi
\end{equation}
is the paraproduct, where $\varphi$ is smooth and satisfies
\begin{equation}\label{cut-off}
\varphi(\xi_1, \xi_2) =
\begin{cases}
1, & |\xi_1| \ll |\xi_2|\text{ and }|\xi_2| \gg 1,\\
0, & |\xi_1| \gg |\xi_2|\text{ or }|\xi_2| \ll 1.
\end{cases}
\end{equation}
\end{df}

\begin{ass}[Assumption 3.1.1 (i) in \cite{AD}]\label{Ass1}
$(h, |\nabla|^{1/2}\psi) \in C([0, T], H^s \times H^{s-1/2}$ and $|\nabla|^{1/2}w \in C([0, T], H^s)$.
\end{ass}

\begin{ass}[Assumption 3.1.1 (ii) in \cite{AD}]\label{Ass2}
\begin{equation}
\sup_{t\in[0,T]} (\|h'\|_{C_*^{\rho-1}} + \|h'\|_{C_*^{-1}}^{1/2}\|h'\|_{H^{-1}}^{1/2}) \le c_{s,\rho}\text{ is small enough}
\end{equation}
\end{ass}

\begin{rem}\label{RemAss2}
By the remark after Assumption 3.1.1 in \cite{AD}, Assumption \ref{Ass2} is guaranteed if $\sup_{t\in[0,T]} \|h\|_{C_*^\rho}$, $\|h(0)\|_{L^2}$ and $\||\nabla|^{1/2}\psi(0)\|_{L^2}$ are small enough.
\end{rem}

\begin{ass}[Assumption 3.1.5 in \cite{AD}]\label{Ass3}
\begin{equation}
\sup_{t\in[0,T]} (\|h(t)\|_{C_*^\rho} + \||\nabla|^{1/2}\psi(t)\|_{C_*^\rho}) \le c_{s,\rho}\text{ is small enough}
\end{equation}
\end{ass}

Under these assumptions we aim to control the growth of the energy
\begin{equation}\label{Es-def}
E_s(t) = \|h(t)\|_{H^s} + \||\nabla|^{1/2}w(t)\|_{H^s}.
\end{equation}
\begin{prop}\label{EneEst}
Let $T > 0$. Let $s > \gamma + 1/2 > 14$ be such that that $2\gamma \notin \Z$. Assume that $\ep \le c_{s,\gamma}$ is small enough and that

(i) Assumption 1 is satisfied, that

(ii) $E_s(0) \le \ep$, and that

(iii) $E_s(t) \le 10\ep$ for all $t \in [0, T]$.

Then there is a constant $C = C_{s,\gamma}$ such that, with $\rho = \gamma - 1/2$,
\begin{equation}
E_s(t)^2 \le 81E_s(0)^2 + C\int_0^t (\|h(\tau)\|_{C_*^\rho} + \||\nabla|^{1/2}\psi(\tau)\|_{C_*^\rho})^2E_s(\tau)^2d\tau.
\end{equation}
\end{prop}
\begin{proof}

We first check the assumptions.
Assumption \ref{Ass1} is already assumed in (i). Since $s > \gamma + 1/2$,
by Remark \ref{Gh-HsRem}, Assumption \ref{Ass2} is also satisfied,
even with $\rho$ replaced by $\gamma$.

Now for Assumption \ref{Ass3}. By Assumption \ref{Ass1}, $|\nabla|^{1/2}\psi \in H^{s-1/2} \sub C_*^{\gamma-1/2}$. Then by Lemma \ref{Gh-Cr},
\begin{align}
\||\nabla|^{1/2}(w - \psi)\|_{H^{s-1/2}} &\le \|T_Bh\|_{H^s}
\lesssim_s \|B\|_{L^\infty}\|h\|_{H^s} \tag{by (A.1.5) of \cite{AD}}\\
&\lesssim \||\nabla|^{1/2}\psi\|_{C_*^3}\|h\|_{H^s}
\lesssim \||\nabla|^{1/2}\psi\|_{H^{s-1/2}}\|h\|_{H^s}.
\label{w-phi-Hs}
\end{align}
Then
\begin{equation}\label{w-Hs=psi-Hs}
\||\nabla|^{1/2}w\|_{H^{s-1/2}} = (1 + O_s(\|h\|_{H^s}))\||\nabla|^{1/2}\psi\|_{H^{s-1/2}}.
\end{equation}
Now for all $t \in [0, T]$, since $\|h(t)\|_{H^s} \le 10\ep \le 10c_{s,\gamma}$ is small enough,
\begin{equation}\label{psi-Hs}
\||\nabla|^{1/2}\psi(t)\|_{C_*^\rho}
\lesssim_{s,\gamma} \||\nabla|^{1/2}\psi(t)\|_{H^{s-1/2}}
\lesssim \||\nabla|^{1/2}w(t)\|_{H^{s-1/2}} \le 10\ep \le 10c_{s,\gamma}
\end{equation}
is also small enough, so Assumption \ref{Ass3} is satisfied.

Then the arguments up to Chapter 3 of \cite{AD} applies. In more detail,
there is a change of variable $(T_ah, |\nabla|^{1/2}w) \mapsto \Phi$ satisfying $E_s/3 \le \|\Phi\|_{H^s} \le 3E_s$ by (3.7.2) and (3.7.3) in \cite{AD} and the quartic energy estimate
\begin{equation}
\|\Phi(t)\|_{H^s}^2 \le \|\Phi(0)\|_{H^s}^2 + C\int_0^t (\|h(\tau)\|_{C_*^\rho}^2 + \||\nabla|^{1/2}\psi(\tau)\|_{C_*^\rho}^2)\|\Phi(\tau)\|_{H^s}^2d\tau
\end{equation}
by (3.7.7) in \cite{AD}, from which the result now follows.
\end{proof}

\section{Strichartz estimates}\label{StrSec}
Let $u = h + i|\nabla|^{1/2}\psi$, whose evolution equation is $u_t + i\Lambda u = N$, where the dispersion relation $\Lambda = |\nabla|^{1/2}$ and (see (6.1) and (4.43) in \cite{Z})
\begin{equation}\label{N-def}
\begin{aligned}
N&=(G(h)-|\nabla|)\psi+\frac i2|\nabla|^{1/2}((1+h'^2)B^2-\psi'^2)=N_2+N_3,\\
N_2&=-|\nabla|(h|\nabla|\psi)-(h\psi')'+\frac i2|\nabla|^{1/2}((|\nabla|\psi)^2-\psi'^2),\\
N_3&=B_3+h'^2B+\frac i2|\nabla|^{1/2}(B^2-(|\nabla|\psi)^2+h'^2B^2),\\
B_3&=B - |\nabla|\psi + |\nabla|(h|\nabla|\psi) + h\psi''.
\end{aligned}
\end{equation}
$N_2$ is a sum of the terms $N_{\mu\nu}=N_{\mu\nu}[u_\mu,u_\nu]$, where $\mu,\nu=\pm$, $u_+ = u$, $u_- = \bar u$,
\begin{equation}\label{Nmunu-def}
\hat N_{\mu\nu}(\xi) = \int_{\xi_1+\xi_2=\xi} m_{\mu\nu}(\xi_1,\xi_2)\hat u_\mu(\xi_1)\hat u_\nu(\xi_2)d\xi_1
\end{equation}
and $m_{\mu\nu}(\xi_1,\xi_2)$ are linear combinations of multipliers in the set
\begin{equation}\label{multipliers}
\left\{\frac{|\xi_1+\xi_2||\xi_2|-(\xi_1+\xi_2)\cdot\xi_2}{\sqrt{|\xi_2|}},
\sqrt{|\xi_1+\xi_2|}\frac{|\xi_1||\xi_2|+\xi_1\cdot\xi_2}{\sqrt{|\xi_1||\xi_2|}}\right\}.
\end{equation}

By Duhamel's formula,
\begin{equation}\label{Duhamel}
u(t)=e^{-it\Lambda}u(0)+u_2(t)+u_3(t),\
u_j(t)=\int_0^t e^{-i(t-\tau)\Lambda}N_j(\tau)d\tau,\ j\in\{2,3\}.
\end{equation}
Using integration by parts in time we get (see (6.4) to (6.6) in \cite{Z})
\begin{equation}\label{IBP}
\begin{aligned}
u_2(t) &= \sum_{\mu,\nu=\pm} \left( Q_{\mu\nu}(t) - e^{-it\Lambda}Q_{\mu\nu}(0) - \int_0^t e^{-i(t-\tau)\Lambda}C_{\mu\nu}(\tau)d\tau \right),\\
\hat Q_{\mu\nu}(\xi, t) &= C\int_{\xi_1+\xi_2=\xi} \frac{m_{\mu\nu}(\xi_1, \xi_2)}{i\Phi_{\mu\nu}(\xi_1, \xi_2)}\hat u_\mu(\xi_1, t)\hat u_\nu(\xi_2, t)d\xi_1,\\
\hat C_{\mu\nu}(\xi, t) &= C\int_{\xi_1+\xi_2=\xi} \frac{m_{\mu\nu}(\xi_1, \xi_2)}{i\Phi_{\mu\nu}(\xi_1, \xi_2)}(\hat u_\mu(\xi_1, t)\hat N_\nu(\xi_2, t) + \hat N_\mu(\xi_1, t)u_\nu(\xi_2, t))d\xi_1
\end{aligned}
\end{equation}
where $N_+ = N$ and $N_- = \bar N$, and $\Phi_{\mu\nu}$, as defined in (\ref{Phi-def}), does not vanish unless $\xi_1$ or $\xi_2$ or $\xi_1 + \xi_2 = 0$, in which case $m_{\mu\nu}(\xi_1, \xi_2) = 0$.

\subsection{Linear estimates}
\begin{lem}[Strichartz estimates]\label{Strichartz}
For $s \in \R$, $\|e^{-it\Lambda}u\|_{L^4_tC_*^s} \lesssim_s \|u\|_{H^{s+3/8}}$.
\end{lem}
\begin{proof}
It suffices to show that $\|P_ke^{-it\Lambda}u\|_{L^4_tL^\infty_x} \lesssim 2^{3k/8}\|u\|_{L^2}$ for $k \in \Z$, where $P_k$ denotes the Littlewood--Paley decomposition. Since $\Lambda$ is homogeneous of degree $1/2$, the scaling $(x, t) \mapsto (2^kx, 2^{k/2}t)$ is a symmetry, so we can assume $k = 0$. Then the result from the standard $t^{-1/2}$ dispersion estimates and the Hardy--Littlewood--Sobolev inequality, see Theorem 2.3 in \cite{Tao} for details.
\end{proof}

\subsection{Bounding the quadratic boundary term $Q_{\mu\nu}$}
To bound the bilinear term $Q_{\mu\nu}$, we need a property of its multiplier.

\begin{df}[Definition C.1--C.2 in \cite{GMS}] A multiplier $m(\xi_1, \xi_2)$ is of class $\mathcal B_s$ if:

\begin{itemize}
\item $m$ is homogeneous of degree $s$,

\item $m$ is smooth outside $\{\xi_1\xi_2(\xi_1 + \xi_2) = 0\}$,

\item Near $\xi_j = 0$ ($j = 1$, 2), $m$ is a smooth function of $|\xi_j|^{1/2}$, $\xi_j/|\xi_j|$ and $\xi_{3-j}$. Near $\xi_1 + \xi_2 = 0$, $m$ is a smooth function of $|\xi_1 + \xi_2|^{1/2}, (\xi_1 + \xi_2)/|\xi_1 + \xi_2|$ and $\xi_1$.
\end{itemize}
If moreover $m$ is supported on $\{|\xi_1| \gtrsim |\xi_2|\}$, we say that it is of class $\tilde B_s$.
\end{df}

\begin{lem}\label{Q-Hr}
Let $s > \gamma + 1 = \rho + 3/2 > 3/2$. Then $\|Q_{\mu\nu}\|_{H^\gamma} \lesssim_{s,\gamma} \|u\|_{C_*^\rho}\|u\|_{H^s}$.
\end{lem}
\begin{proof}
By Section 3 of \cite{GMS}, $m_{\mu\nu} \in \mathcal B_{3/2}$
and $\Phi_{\mu\nu} \in \mathcal B_{1/2}$. Hence $m_{\mu\nu}/\Phi_{\mu\nu} \in \mathcal B_1$, and can be decomposed as $m_1 + m_2$ where $m_j$ captures the contribution where the frequency of the $j$-th slot is bounded below by a (small) constant times the frequency of the other slot. Thus, for example, $m_1 \in \mathcal{\tilde B}_1$.

Let $Q_{\mu\nu,j}$ be the corresponding bilinear product. Then the multiplier of $\sqrt{1-\Delta}^\gamma Q_{\mu\nu,1}(\sqrt{1-\Delta}^{-\gamma-1}\cdot, \cdot)$ is of class $\mathcal B_0$ and by Theorem C.1 (i) of \cite{GMS} satisfies
\begin{equation}\label{Q-L2}
\|\sqrt{1-\Delta}^\gamma Q_{\mu\nu,1}(\sqrt{1-\Delta}^{-\gamma-1}f, g)\|_{L^2} \lesssim_{p,q} \|f\|_{L^o}\|g\|_{L^q}
\end{equation}
where $2 < p, q < \infty$ and $1/p + 1/q = 1/2$. Then
\begin{equation}
\|Q_{\mu\nu,1}(u, u)\|_{H^\gamma} \lesssim_{\gamma,p,q} \|u\|_{H^{\gamma+1,p}}\|u\|_{L^q}
\end{equation}
where $H^{s,p} = F^s_{p,2}$ is the Bessel potential space (see Section 2.3 of \cite{Tr}). Exchanging two slots gives the same bound for $Q_{\mu\nu,2}$, so the same bound holds for $Q_{\mu\nu}$.

Now let $s' = (s + \gamma + 1)/2 \in (\gamma + 1, s)$. Let $p = 4(s' - \rho)/3 > 2$. Then $H^{\gamma+1,p} = F^{\gamma+1}_{p,2} \supset B^{\gamma+1}_{p,2}$ interpolates between $B^\rho_{\infty,\infty} = C_*^\rho$ and $B^{s'}_{2,4/p} \supset H^s$, and $L^q$ interpolates in the same way between $L^2$ and $L^\infty$, we have
\begin{equation}
\|Q_{\mu\nu}\|_{H^\gamma} \lesssim_{s,\gamma} \|u\|_{C_*^\rho}\|u\|_{L^2} + \|u\|_{H^s}\|u\|_{L^\infty} \lesssim \|u\|_{C_*^\rho}\|u\|_{H^s}.
\end{equation}
\end{proof}

Now we let
\begin{equation}\label{FG-def}
F_s = \sup_{t\in[0,T]} \|u(t)\|_{H^s},\quad G = \|u\|_{L^4([0,T])C_*^\rho}.
\end{equation}

\begin{prop}\label{Q-L4Loo}
Let $s > \rho + 3/2 > 3/2$. Then
\begin{equation}
\|Q_{\mu\nu}(t) - e^{-it\Lambda}Q_{\mu\nu}(0)\|_{L_\tau^4([0,T])C_*^\rho}
\lesssim_{s,\rho} F_s^2 + GF_s.
\end{equation}
\end{prop}
\begin{proof}
On one hand,
\(
\|Q_{\mu\nu}\|_{C_*^\rho} \lesssim_\rho \|Q_{\mu\nu}\|_{H^{\rho+1/2}}
\lesssim_{s,\rho} \|u\|_{C_*^\rho}\|u\|_{H^s}
\)
by Lemma \ref{Q-Hr}, so the contribution of $Q_{\mu\nu}$ is controlled by $GF_s$.
On the other hand, by Lemma \ref{Strichartz} and Lemma \ref{Q-Hr},
the contibution of $e^{-it\Lambda}Q_{\mu\nu}(0)$ is controlled by
\begin{equation}
\|Q_{\mu\nu}(0)\|_{H^{\rho+1/2}} \lesssim_{s,\rho} \|u(0)\|_{C_*^\rho}\|u(0)\|_{H^s} \lesssim_{s,\rho} F_s^2
\end{equation}
from which we get the result.
\end{proof}

\subsection{Bounding the cubic bulk term $u_3$}
Note that in (\ref{N-def}), $B_3 = B - B_{\le2}$ as defined in (2.6.1) in \cite{AD}. Thus we have:

\begin{lem}[Proposition 2.6.1 in \cite{AD}]\label{B3-Hs}
Let $s, \gamma, \mu$ be such that $s - 1/2 > \gamma > 14$, $s \ge \mu \ge 5$ and $2\gamma \notin \Z$. Then for all $(h, |\nabla|^{1/2}\psi) \in H^{s+1/2} \times (C_*^{\gamma-1/2} \cap H^\mu)$ such that $\|h\|_{C_*^\gamma} \le c_{s,\gamma,\mu}$ is small enough,
\begin{equation}
\|B_3\|_{H^{\mu-1}} \lesssim_{s,\gamma,\mu} \|h\|_{C_*^\gamma}(\||\nabla|^{1/2}\psi\|_{C_*^{\gamma-1/2}}\|h\|_{H^s} + \|h\|_{C_*^\gamma}\||\nabla|^{1/2}\psi\|_{H^\mu}).
\end{equation}
\end{lem}

\begin{lem}\label{N3-Hs}
Assume, besides the assumptions of Lemma \ref{B3-Hs}, further that $s \ge \mu + 1/2$ and $\|h\|_{H^s} \le c_s$ is small enough. Then
\(
\|N_3\|_{H^{\mu-1}} \lesssim_{s,\gamma,\mu} \|u\|_{C_*^\gamma}^2\|u\|_{H^s}.
\)
\end{lem}
\begin{proof}
Recall from (\ref{N-def}) that
\begin{equation}\label{N3-def}
N_3=B_3+h'^2B+\frac i2|\nabla|^{1/2}(B^2-(|\nabla|\psi)^2+h'^2B^2).
\end{equation}
The bound for $B_3$ has already been shown. For the other terms we use
the Sobolev multiplication theorem and the bounds in Section \ref{GhSec}.
In the following we write $X$ for $\|u\|_X$ and omit the dependence of constants on $s, \gamma, \mu$ to avoid clutter.
\begin{align*}
\|h'^2B\|_{H^{\mu-1}} &\lesssim \|h'\|_{L^\infty}(\|h'\|_{L^\infty}\|B\|_{H^{\mu-1}} + \|h'\|_{H^{\mu-1}}\|B\|_{L^\infty})\\
&\lesssim W^{1,\infty}(W^{1,\infty}H^{\mu-1/2} + H^\mu C_*^3)
\lesssim (C_*^3)^2H^\mu,\\
\|B^2 - (|\nabla|\psi)^2\|_{H^{\mu-1/2}}
&\lesssim \|B + |\nabla|\psi\|_{H^{\mu-1/2}}\|B - |\nabla|\psi\|_{L^\infty}\\
&+ \|B + |\nabla|\psi\|_{L^\infty}\|B - |\nabla|\psi\|_{H^{\mu-1/2}}\\
&\lesssim H^\mu (C_*^4)^2 + C_*^3(C_*^3H^s + C_*^4H^\mu)
\lesssim (C_*^4)^2H^s,\\
\|h'^2B^2\|_{H^{\mu-1/2}} &\lesssim \|h'B^2\|_{H^{\mu-1/2}}
\tag{by smallness of $\|h'\|_{H^{\mu-1/2}}$}\\
&\lesssim \|B\|_{L^\infty}(\|h'\|_{L^\infty}\|B\|_{H^{s-1/2}} + \|h'\|_{H^{s-1/2}}\|B\|_{L^\infty})\\
&\lesssim C_*^3(W^{1,\infty}H^\mu + H^{\mu+1/2}C_*^3)
\lesssim (C_*^3)^2H^s.
\end{align*}
\end{proof}

\begin{prop}\label{u3-L4Loo}
Let $s > \rho + 2 > 16$. Assume $\|h\|_{H^s} \le c_{s,\rho}$ is small enough. Then
\begin{equation}
\|u_3\|_{L^4([0,T])C_*^\rho} \lesssim_{s,\rho} \sqrt TG^2F_s.
\end{equation}
\end{prop}
\begin{proof}
Pick $\gamma \in [\rho + 1/2, s - 3/2]$ such that $2\gamma \notin \Z$. Then
\begin{align*}
\|u_3\|_{L^4([0,t])C_*^\rho} &\le \int_0^t \|e^{-i(t-\tau)\Lambda}N_3(\tau)\|_{L_t^4([\tau,T])C_*^\rho}d\tau \tag{by the triangle inequality}\\
&\lesssim_\rho \|N_3\|_{L^1([0,T])H_x^\gamma}
\tag{by Lemma \ref{Strichartz}}\\
&\le \sqrt T\|N_3\|_{L^2([0,T])H_x^\gamma}
\tag{by the Cauchy--Schwarz inequality}\\
&\lesssim_{s,\rho} \sqrt TG^2F_s.
\tag{by Lemma \ref{N3-Hs} and H\"older's inequality}
\end{align*}
\end{proof}

\subsection{Bounding the cubic bulk term $C_{\mu\nu}$}
\begin{prop}\label{C-L4Loo}
Let $s > \rho + 3 > 6.5$. Then
\begin{equation}
\left\| \int_0^t e^{-i(t-\tau)\Lambda}C_{\mu\nu}(\tau)d\tau \right\|_{L_t^4([0,T])C_*^\rho} \lesssim_{s,\rho} \sqrt TG^2F_s.
\end{equation}
\end{prop}
\begin{proof}
Aa in the proof of Proposition \ref{u3-L4Loo}, it suffices to show that
\begin{equation}\label{C-Hr3}
\|C_{\mu\nu}\|_{H^{\rho+1/2}} \lesssim_{s,\rho} \|u\|_{C_*^\rho}^2\|u\|_{H_s}.
\end{equation}
By (\ref{Q-L2}),
\begin{equation}\label{C-Hr2}
\|C_{\mu\nu}\|_{H^{\rho+1/2}} \lesssim_{\rho,p,q} \|N\|_{L^q}\|U\|_{H^{\rho+3/2,p}} + \|U\|_{L^q}\|N\|_{H^{\rho+3/2,p}}
\end{equation}
where $1/p + 1/q = 1/2$. By Lemma \ref{Gh-Hs2}, Lemma \ref{Gh-Cr2} and the expression for $N$ in (\ref{N-def}),
\begin{equation}\label{N-Lq}
\|N\|_{L^q} \le \|N\|_{L^2}^{2/q}\|N\|_{L^\infty}^{1-2/q}
\lesssim \|u\|_{C_*^\rho}^{2-2/q}\|u\|_{H^s}^{2/q}.
\end{equation}
By interpolation (see the proof of Lemma \ref{Q-Hr}), for $p = 2(s - \rho)/3 > 2$,
\begin{align}
\label{u-Hr}
\|u\|_{H^{\rho+3/2,p}} &\le \|u\|_{H^{\rho+3,p}}
\lesssim_{s,\rho} \|u\|_{C_*^\rho}^{2/q}\|u\|_{H^s}^{1-2/q},\\
\label{u-Lq}
\|u\|_{L^q} &\le \|u\|_{L^2}^{2/q}\|u\|_{L^\infty}^{1-2/q}
\le \|u\|_{H^s}^{2/q}\|u\|_{C_*^\rho}^{1-2/q}.
\end{align}
Again by interpolation, Lemma \ref{Gh-Cr2} and Lemma \ref{Gh-Hs2},
\begin{equation}\label{N-Hr}
\|N\|_{H^{\rho+3/2,p}} \lesssim_{s,\rho} \|N\|_{C_*^{\rho-3/2}}^{2/q}\|N\|_{H^{s-3/2}}^{1-2/q} \lesssim_{s,\rho} \|u\|_{C_*^\rho}^{1+2/q}\|u\|_{H^s}^{1-2/q}.
\end{equation}
Combining all the four bounds above shows the result.
\end{proof}

\begin{prop}\label{StrEst}
Let $s > \rho + 3 > 17$. Then $G \lesssim_{s,\rho} F_s + F_s^2 + GF_s + \sqrt TG^2F_s$.
\end{prop}
\begin{proof}
This follows from (\ref{Duhamel}), Lemma \ref{Strichartz}, Proposition \ref{Q-L4Loo}, Proposition \ref{u3-L4Loo} and Proposition \ref{C-L4Loo}.
\end{proof}

\section{The Euclidean case: proof of Theorem \ref{Thm1}}\label{Thm1prf}
\begin{proof}[Proof of Theorem \ref{Thm1}]
Since the water wave equation is locally wellposed (see \cite{Wu2DL} for example), we only need to show {\it a priori} estimates that can be closed.

Recall from (\ref{Es-def}) and (\ref{FG-def}) that
\begin{align}
E_s(t) &= \|h(t)\|_{H^s} + \||\nabla|^{1/2}w(t)\|_{H^s},\\
F_s &= \sup_{t\in[0,T]} \|u(t)\|_{H^s},\quad G = \|u\|_{L^4([0,T])C_*^\rho}.
\end{align}
We assume $s > \rho + 3.5 > 17.5$ and the following assumptions:

\begin{enumerate}
\item $(h, |\nabla|^{1/2}\psi) \in C([0, T], H^s \times H^{s-1/2}$ and $|\nabla|^{1/2}w \in C([0, T], H^s)$,

\item $E_s(0) \le \ep$,

\item $E_s(t) \le 10\ep$ for all $t \in [0, T]$, and

\item $G \le A\ep$ (where $A$ is a constant to be determined later),
\end{enumerate}
of which the last two need to be closed. By Proposition \ref{EneEst},
\begin{equation}\label{Es2-bound}
E_s(t)^2 \le 81E_s(0)^2 + C_{s,\rho}G^2\ep^2
\le 81\ep^2 + C_{s,\rho}A^2\ep^4
\end{equation}
so
\begin{equation}\label{Es-bound}
E_s(t) \le 9\ep + C_{s,\rho}A\ep^2 \le 9.5\ep
\end{equation}
provided that $\ep \le 0.5/C_{s,\rho}A$, closing the bootstrap assumption on $E_s$.

For $G$ we have, by Proposition \ref{StrEst} (with $s - 1/2$ in place of $s$),
\begin{equation}\label{G-bound}
G \lesssim_{s,\rho} F_{s-1/2} + F_{s-1/2}^2 + GF_{s-1/2} + \sqrt TG^2F_{s-1/2}.
\end{equation}
By (\ref{psi-Hs}),
\begin{equation}\label{F-bound}
F_{s-1/2} \lesssim_s \sup_{[0,T]} E_s \le 10\ep.
\end{equation}
Putting (\ref{F-bound}) and bootstrap assumption on $G$ ($G \le A\ep$) in (\ref{G-bound}) gives
\begin{equation}\label{G-bound2}
G \le C_{s,\rho}(\ep + \ep^2 + A\ep^2 + \sqrt TA^2\ep^3).
\end{equation}
Now let $A = A_{s,\rho} = \max(8C_{s,\rho}, 2)$. Then
\begin{equation}
C_{s,\rho}A_{s,\rho}\ep^2 < \ep/2 \le A_{s,\rho}\ep/4, \ C_{s,\rho}\ep^2 < \ep/4 \le A_{s,\rho}\ep/8, \ C_{s,\rho}\ep \le A_{s,\rho}\ep/8
\end{equation}
so (\ref{G-bound2}) becomes
\begin{equation}\label{G-bound3}
G < A_{s,\rho}\ep/2 + C_{s,\rho}A_{s,\rho}^2\sqrt T\ep^3.
\end{equation}
Now we choose
\begin{equation}\label{T-def}
T = \frac{1}{9C_{s,\rho}^2A_{s,\rho}^2\ep^4}
\end{equation}
so that $G < A_{s,\rho}\ep/2 + A_{s,\rho}\ep/3 = 5A_{s,\rho}\ep/6$, closing the bootstrap assumption on $G$.
\end{proof}

\section{The periodic case: proof of Theorem \ref{Thm2}}\label{Thm2prf}
For water waves with a period of $R$, the energy estimate is not affected,
while the Strichartz estimate becomes the following:

\begin{lem}\label{Strichartz-R}
For $s \in \R$, $\|e^{-it\Lambda}u\|_{L^4_t([0,T])C_*^s} \lesssim_s \sqrt[4]{1 + T/R}\|u\|_{H^{s+3/8}}$.
\end{lem}
\begin{proof}
Since wave packets of frequencies $\xi \approx 2^k$ travel at speed $\Lambda'(\xi) \approx \xi^{-1/2} \approx 2^{-k/2}$, they do not reach the boundary and wrap around until time $\approx 2^{k/2}R$. Hence for time $T \lesssim 2^{k/2}R$, the period estimate is the same as the Euclidean one.
Longer time periods can be partitioned into $O(1 + 2^{-k/2}T/R)$ segments of length $2^{k/2}R$. Since the evolution conserves the $L^2$ norm, in each segment the estimate above holds, which then adds up to
\begin{equation}
\begin{aligned}
\|P_ke^{-it\Lambda}u\|_{L_t^4([0,T])L_x^\infty}
&\lesssim 2^{3k/8}\sqrt[4]{1 + 2^{-k/2}T/R}\|u\|_{L^2}\\
&\lesssim 2^{k/4 + \max(k,0)/8}\sqrt[4]{1 + T/R}\|u\|_{L^2}.
\end{aligned}
\end{equation}
Summing over $k \in \Z$ shows the result.
\end{proof}

Now we are ready to prove Theorem \ref{Thm2}.
\begin{proof}[Proof of Theorem \ref{Thm2}]
Using the same notation $E_s(t)$, $F_s$ and $G$ as before we have
\begin{align}
\label{Es-bound-R}
E_s(t)^2 &\le 81E_s(0)^2 + C_{s,\rho}\ep^2G^2,\\
\label{G-bound-R}
G &\lesssim_{s,\rho} \sqrt[4]{1 + T/R}(F_{s-1/2} + F_{s-1/2}^2 + \sqrt TG^2F_{s-1/2}) + GF_{s-1/2},\\
\label{F-bound-R}
F_{s-1/2} &\lesssim_s \sup_{[0,T]} E_s.
\end{align}

If $R > \ep^{-4}$ then we use the same bootstrap assumptions as the Eulidean case, which show a lifespan of $T \approx_{s,\rho} \ep^{-4}$.
Since $T/R \lesssim_{s,\rho} 1$, the extra factor $\sqrt[4]{1 + T/R}$ can be safely ignored.

If $\ep^{-2} < R \le \ep^{-4}$ then our bootstrap assumptions are $E_s(t) \lesssim \ep$ ($t \in [0, T]$) and $G \lesssim \sqrt\ep/\sqrt[8]{R}$. Since $G \lesssim \ep^{3/4}$ is small enough, the assumption on $E_s$ can be closed. For $G$ the dominant terms are (note that $F_{s-1/2} \lesssim_s \ep$ by (\ref{F-bound-R}))
\begin{equation}\label{G-dom}
\sqrt[4]{1 + T/R}\cdot\ep\text{ and }\sqrt[4]{1 + T/R}\cdot\ep\sqrt TG^2
= \sqrt[4]{1 + T/R}\cdot\ep^2\sqrt T/\sqrt[4]{R}.
\end{equation}
An easy computation shows that the estimate can be closed up to the lifespan
\begin{equation}\label{TR-def}
T \approx_{s,\rho} \sqrt R/\ep^2.
\end{equation}

If $1 \le R \le \ep^{-2}$ then we resort to \cite{Wu2DL}, noting that that result carries over to the periodic case.
\end{proof}

\end{document}